\documentclass[12pt]{amsart}
\usepackage{amssymb}
\usepackage{graphicx}
\usepackage{amscd}
\usepackage{amsmath}
\usepackage[latin1]{inputenc}
\usepackage[T1]{fontenc}

\newtheorem{theorem}{Theorem}

\newtheorem{corollary}{Corollary}

\newtheorem{proposition}{Proposition}

\numberwithin{equation}{section}

\DeclareMathOperator{\supp}{supp}
\DeclareMathOperator{\Dim}{Dim}
\input{tcilatex}

\begin{document}
\title[An other upper bounding for packing dimension ]{Upper bounding for
packing dimension in vectorial multifractal formalism}
\author{L. Ben Youssef}
\address{Current address : L. Ben Youssef. ISCAE. University of Manouba.
Tunisia.}
\email{Leila.BenYoussef@iscae.rnu.tn}
\date{January, 2011}
\thanks{This paper is in final form and no version of it will be submitted
for publication elsewhere.}

\begin{abstract}
We establish an other upper bounding for packing dimension in the framework
of the vectorial multifractal formalism that is in some cases finer than
that established by J. Peyrière.
\end{abstract}

\maketitle

\section{Introduction}

The multifractal analysis was developed around 1980, following the work of
B. Mandelbrot \cite{Man0, Man}, when he studied the multiplicative cascades
for energy dissipation in a context of turbulence. In 1992, G. Brown, G.
Michon and J. Peyrière \cite{G.B} have established the first general and
rigorous theorems of the multifractal formalism. Their work prompted the
three past decades, several mathematicians \cite{J.P1, S.J, O.L1, S.J2,
F.B3, J.P2} ..., to develop their research in various contexts by
generalizing or improving the multifractal formalism.

In this paper, we take place in the framework of the vectorial multifractal
formalism introduced by J. Peyrière \cite{J.P2} in 2004. We recall at the
end of this paragraph the results of this formalism that we are going to use
later. In the second section, we give an other upper bounding for packing
dimension \cite{C.T} of the set 
\begin{equation*}
X_{\chi }\left( \alpha ,E\right) =\left\{ x\in \mathbb{X};\text{\thinspace }%
\underset{r\rightarrow 0}{\lim \sup }\frac{\left\langle q,\chi
(x,r)\right\rangle }{\log r}\leq \left\langle q,\alpha \right\rangle ,\text{ 
}\forall q\in E\right\} ,
\end{equation*}%
where $\mathbb{X}$ is a metric space verifying the Besicovitch covering
property, $E$ is a subset of a separable real Banach space $\mathbb{E}$, $%
\chi $ is a function from $\mathbb{X\times }\left] 0,1\right] $ to the dual $%
\mathbb{E}^{\prime }$ and $\alpha \in \mathbb{E}^{\prime }.$

In the third section, we present some situations where our inequality is
finer than that made by J. Peyrière in \cite{J.P2}.

In what follows, we recall the vectorial multifractal formalism introduced
by J. Peyrière in \cite{J.P2}.

For $A\subset \mathbb{X}$, $q\in \mathbb{E}$, $t\in \mathbb{R}$ and $%
\varepsilon \in \left] 0,1\right] ,$ we set 
\begin{equation*}
\overline{P}_{\chi ,\varepsilon }^{q,t}(A)=\sup \left\{ \underset{i}{\sum }%
r_{i}^{t}e^{\left\langle q,\chi (x_{i},r_{i})\right\rangle }\right\} ,
\end{equation*}%
where the supremum is taken over all the centered $\varepsilon -$packing $%
(B\left( x_{i},r_{i}\right) )_{i\in I}$ of $A$. \newline
Then, we set%
\begin{equation*}
\overline{P}_{\chi }^{q,t}(A)=\underset{\varepsilon \rightarrow 0}{\lim }%
\overline{P}_{\chi ,\varepsilon }^{q,t}(A)
\end{equation*}%
and 
\begin{equation*}
P_{\chi }^{q,t}(A)=\inf \left\{ \underset{i}{\sum }\overline{P}_{\chi
}^{q,t}(A_{i});\text{ }A\subset \underset{i}{\cup }A_{i}\right\} .
\end{equation*}%
It is clear that 
\begin{equation}
P^{q,t}(A)=\inf \left\{ \underset{i}{\sum }\overline{P}_{\chi }^{q,t}(A_{i});%
\text{ }A=\underset{i}{\cup }A_{i}\right\}  \label{eq1v}
\end{equation}%
and 
\begin{equation}
P_{\chi }^{q,t}(A)=\inf \left\{ \underset{i}{\sum }\overline{P}_{\chi
}^{q,t}(A_{i}),\text{ }\left( \underset{i}{\cup }A_{i}\right) \ \text{is a
partition of }A\right\} .  \label{eq2v}
\end{equation}%
We denote by $\Delta _{\chi }^{q}(A)$ and $\Dim_{\chi }^{q}(A)$ the
dimensions of $A$ characterized by 
\begin{equation*}
\overline{P}_{\chi }^{q,t}(A)=\left\{ 
\begin{array}{c}
+\infty ,\text{ if }t<\Delta _{\chi }^{q}(A), \\ 
0,\text{ if }t>\Delta _{\chi }^{q}(A),%
\end{array}%
\right.
\end{equation*}%
and 
\begin{equation*}
P_{\chi }^{q,t}(A)=\left\{ 
\begin{array}{c}
+\infty \text{, if }t<\Dim_{\chi }^{q}(A), \\ 
0,\text{ if }t>\Dim_{\chi }^{q}(A).%
\end{array}%
\right.
\end{equation*}

For $\mathbb{X}=\mathbb{R}^{d}$, $\mathbb{E}=\mathbb{R}$, $\mu $ a Borel
probability measure on $\mathbb{R}^{d}$, and considering the function $\chi $
defined by%
\begin{equation*}
\left\langle q,\chi (x_{i},r_{i})\right\rangle =q\log \mu (B(x_{i},r_{i})).
\end{equation*}%
for all centered $\varepsilon -$packing $(B\left( x_{i},r_{i}\right) )_{i\in
I}$ of $A,$ we found the formalism introduced by L. Olsen \cite{O.L1}, in
particular we get 
\begin{equation*}
\Delta _{\chi }^{q}(A)=\Delta _{\mu }^{q}(A)\text{ \ \ and \ \ \ }\Dim%
_{\varkappa }^{q}(A)=\Dim_{\mu }^{q}(A).
\end{equation*}%
Furthermore, note also that for the trivial case $\chi =0$, we obtain the
prepacking dimension and the packing dimension of $A$, ie 
\begin{equation*}
\Delta _{\chi }^{q}(A)=\Delta (A)\text{ \ \ and \ \ \ }\Dim_{\varkappa
}^{q}(A)=\Dim(A).
\end{equation*}

The following proposition and theorem are established in \cite{J.P2}.

\begin{proposition}
\label{prop1v}Write $\Lambda _{\chi }(q)=\Delta _{\chi }^{q}(\mathbb{X})$
and $B_{\chi }(q)=\Dim_{\chi }^{q}(\mathbb{X}).$ Then\newline
i. $B_{\chi }\leq \Lambda _{\chi }.$\newline
ii. The functions $\Lambda _{\chi }:q\mapsto \Lambda _{\chi }(q)$ and $%
B_{\chi }:q\mapsto B_{\chi }(q)$ are convex.
\end{proposition}

\begin{theorem}
\label{th1v}For $\alpha \in \mathbb{E}^{\prime }$ and $E\subset \mathbb{E}$
we set 
\begin{equation*}
X_{\chi }\left( \alpha ,E\right) =\left\{ x\in \mathbb{X};\text{\thinspace }%
\underset{r\rightarrow 0}{\lim \sup }\frac{\left\langle q,\chi
(x,r)\right\rangle }{\log r}\leq \left\langle q,\alpha \right\rangle ,\text{ 
}\forall q\in E\right\} ,
\end{equation*}%
then 
\begin{equation*}
\Dim(X_{\chi }\left( \alpha ,E\right) )\leq \underset{q\in E}{\inf }%
(\left\langle q,\alpha \right\rangle +B_{\chi }(q)).
\end{equation*}
\end{theorem}

\section{An other upper bounding for $\Dim(X_{\protect\chi }\left( \protect%
\alpha ,E\right) )$}

Let $\varepsilon >0$ be a real number and $k\geq 1$ be an integer. A family $%
(B(x_{i},r_{i}))_{i\in I}$ is called a centered $\varepsilon -k-$Besicovich
\ packing of a set $A$ when $I=I_{1}\cup ...\cup I_{s}$ with \ $1\leq s\leq
k $ and $\left( B(x_{i},r_{i})\right) _{i\in I_{j}}$ a centered $\varepsilon
- $packing of $A$ for all $1\leq j\leq s$.\newline
Let $(u_{\varepsilon })_{\varepsilon >0}$ be a decreasing family of numbers
such that $\varepsilon \leq u_{\varepsilon }$ and $\underset{\varepsilon
\rightarrow 0}{\lim }u_{\varepsilon }=0$. \newline
For $q\in E$, $A$ $\subset \mathbb{X}$ and $(B(x_{i},r_{i}))_{i\in I}$ a
centered $\varepsilon -$packing of $A$, we consider all the families $%
(B(y_{i},\delta _{i}))_{i\in I}$ that are centered $u_{\varepsilon }-k-$%
Besicovitch packing of $A$ and we set 
\begin{equation*}
L_{\varepsilon ,(B(x_{i},r_{i}))_{i\in I}}^{q,k}(A)=\inf \left( \underset{%
i\in I}{\sup }\left( \frac{\left\langle q,\chi (y_{i},\delta
_{i})\right\rangle }{\log r_{i}}\right) \right) ,
\end{equation*}%
where the infimum is taken over all the centered $u_{\varepsilon }-k-$%
Besicovich packing $(B(y_{i},\delta _{i}))_{i\in I}$ of $A.$\newline
It is clear that 
\begin{equation}
L_{\varepsilon ,(B(x_{i},r_{i}))_{i\in I}}^{q,k}(A)\leq \underset{i\in I}{%
\sup }\left( \frac{\left\langle q,\chi (x_{i},r_{i})\right\rangle }{\log
r_{i}}\right) .  \label{eq5v}
\end{equation}%
Write 
\begin{equation*}
L_{\varepsilon }^{q,k}(A)=\sup \left\{ L_{\varepsilon
,(B(x_{i},r_{i}))_{i\in I}}^{q,k}(A)\right\} ,
\end{equation*}%
where the supremum is taken over all the centered $\varepsilon -$packing $%
(B\left( x_{i},r_{i}\right) )_{i\in I}$ of $A.$\newline
We remark that for $\varepsilon <\varepsilon ^{\prime }$, $L_{\varepsilon
^{\prime }}^{q,k}(A)>L_{\varepsilon }^{q,k}(A)$, then we define 
\begin{equation*}
L^{q,k}(A)=\underset{\varepsilon \rightarrow 0}{\lim }L_{\varepsilon
}^{q,k}(A).
\end{equation*}%
As the sequence $\left( L^{q,k}(A)\right) _{k}$ is decreasing, write 
\begin{equation*}
L^{q}(A)=\underset{k\rightarrow +\infty }{\lim }L^{q,k}(A).
\end{equation*}

Before giving our new inequality involving $\Dim(X_{\chi }\left( \alpha
,E\right) )$, we first illustrate our main idea on the set $A^{\left\langle
q,\alpha \right\rangle }$ defined for $\alpha \in \mathbb{E}^{\prime }$, $%
q\in E$ and $r_{0}>0$ by 
\begin{equation}
A^{\left\langle q,\alpha \right\rangle }=\left\{ x\in \mathbb{X};\text{%
\thinspace }r^{\left\langle q,\alpha \right\rangle }\leq e^{\left\langle
q,\chi (x,r)\right\rangle },\text{ for }r<r_{0}\right\} .  \label{pA}
\end{equation}

\begin{proposition}
\label{prop2v}\ 
\begin{equation*}
L^{q}(A^{\left\langle q,\alpha \right\rangle })\leq \left\langle q,\alpha
\right\rangle .
\end{equation*}
\end{proposition}

\begin{proof}
let $\varepsilon <r_{0}$ and $(B(x_{i},r_{i}))_{i\in I}$ a centered $%
\varepsilon -$packing of $A^{\left\langle q,\alpha \right\rangle }$. Thanks
to the characteristic property of $A^{\left\langle q,\alpha \right\rangle }$
(\ref{pA}), it comes for all $i\in I,$ 
\begin{equation*}
\frac{\left\langle q,\chi (x_{i},r_{i})\right\rangle }{\log r_{i}}\leq
\left\langle q,\alpha \right\rangle ,
\end{equation*}%
hence 
\begin{equation*}
\underset{i\in I}{\sup }\frac{\left\langle q,\chi (x_{i},r_{i})\right\rangle 
}{\log r_{i}}\leq \left\langle q,\alpha \right\rangle ,
\end{equation*}%
from the inequality (\ref{eq5v}), we deduce that 
\begin{equation*}
L_{\varepsilon ,(B(x_{i},r_{i}))_{i\in I}}^{q,k}(A^{\left\langle q,\alpha
\right\rangle })\leq \left\langle q,\alpha \right\rangle ,
\end{equation*}%
while considering the supremum over all centered $\varepsilon -$packing, it
results that 
\begin{equation*}
L_{\varepsilon }^{q,k}(A^{\left\langle q,\alpha \right\rangle })\leq
\left\langle q,\alpha \right\rangle .
\end{equation*}%
Letting $\varepsilon \rightarrow 0,$ we obtain that 
\begin{equation*}
L^{q,k}(A^{\left\langle q,\alpha \right\rangle })\leq \left\langle q,\alpha
\right\rangle ,
\end{equation*}%
then letting $k\rightarrow +\infty ,$ it comes that 
\begin{equation*}
L^{q}(A^{\left\langle q,\alpha \right\rangle })\leq \left\langle q,\alpha
\right\rangle .
\end{equation*}
\end{proof}

\begin{theorem}
\label{thIv}Let $\alpha \in \mathbb{E}^{\prime }$ and $q\in E$. \newline
For $t<0$ we set $\Phi _{q}(t)=\inf \left\{ \gamma >0;\text{ }t\left\langle
q,\alpha \right\rangle >B_{\chi }((\gamma -t)q)\right\} $. Then, 
\begin{equation*}
\Dim(A^{\left\langle q,\alpha \right\rangle })\leq \Phi
_{q}(t)L^{q}(A^{\left\langle q,\alpha \right\rangle }).
\end{equation*}
\end{theorem}

\begin{proof}
For $t<0$ and $\gamma >0$ such that $t\left\langle q,\alpha \right\rangle
>B_{\chi }((\gamma -t)q,$ it is clear that $P_{\chi }^{(\gamma
-t)q,t\left\langle q,\alpha \right\rangle }(\mathbb{X})=0,$ Then $P_{\chi
}^{(\gamma -t)q,\left\langle q,\alpha \right\rangle }(A^{\left\langle
q,\alpha \right\rangle })=0.$\newline
From the equality (\ref{eq1v}), we write 
\begin{equation}
A^{\left\langle q,\alpha \right\rangle }=\underset{m\in M}{\cup }A_{m}
\label{eqp0v}
\end{equation}%
such that for all $m\in M,$%
\begin{equation}
\overline{P}_{\chi }^{(\gamma -t)q,t\left\langle q,\alpha \right\rangle
}(A_{m})<+\infty .  \label{eqp1v}
\end{equation}%
Let $\lambda >L^{q}\left( A^{\left\langle q,\alpha \right\rangle }\right) ,$
let us prove first that for all $m$ $\in M$, 
\begin{equation*}
\triangle (A_{m})\leq \gamma \lambda .
\end{equation*}%
As $A_{m}\subset A^{\left\langle q,\alpha \right\rangle }$ and $\lambda
>L^{q}(A_{m}),$ then there exist an integer $k\geq 1$ and a real number $%
\varepsilon _{0}<r_{0}$ such that for all $\varepsilon <\varepsilon _{0},$ 
\begin{equation*}
L_{\varepsilon }^{q,k}(A_{m})<\lambda .
\end{equation*}%
It comes that for all tout centered $\varepsilon -$packing $(B(x_{i},r_{i}))$
of $A_{m},$ there exists a centered $u_{\varepsilon }-k-$Besicovitch packing 
$(B(y_{i},\delta _{i}))_{i\in I}$ of $A_{m}$ such that for all $i\in I,$ 
\begin{equation*}
\dfrac{\left\langle q,\chi (y_{i},\delta _{i})\right\rangle }{\log r_{i}}%
<\lambda ,
\end{equation*}%
so that 
\begin{equation}
r_{i}^{\lambda }<e^{\left\langle q,\chi (y_{i},\delta _{i})\right\rangle }.
\label{eqiv}
\end{equation}%
Thanks to the characteristic property of $A^{\left\langle q,\alpha
\right\rangle }$ (\ref{pA}), it comes that 
\begin{equation}
\delta _{i}^{\left\langle q,\alpha \right\rangle }<e^{\left\langle q,\chi
(y_{i},\delta _{i})\right\rangle }.  \label{eqiiv}
\end{equation}%
Thus from the inequalities (\ref{eqiv}) and (\ref{eqiiv}), we obtain that
for all $\gamma >0$ and $t<0,$ 
\begin{equation*}
r_{i}^{\gamma \lambda }\leq e^{(\gamma -t)\left\langle q,\chi (y_{i},\delta
_{i})\right\rangle }\delta _{i}^{t\left\langle q,\alpha \right\rangle }.
\end{equation*}%
Using the equality $I=I_{1}\cup ...\cup I_{s}$ with $1\leq s\leq k$ and $%
\left( B(x_{i},r_{i})\right) _{i\in I_{j}}$ a centered $\varepsilon -$%
packing of $A_{m}$ for all $1\leq j\leq s$, it follows that 
\begin{equation*}
\underset{i\in I}{\sum }r_{i}^{\gamma \lambda }\leq \underset{i\in I}{\sum }%
e^{(\gamma -t)\left\langle q,\chi (y_{i},\delta _{i})\right\rangle }\delta
_{i}^{t\left\langle q,\alpha \right\rangle }=\underset{j=1}{\overset{s}{\sum 
}}\underset{i\in I_{j}}{\sum }e^{(\gamma -t)\left\langle q,\chi
(y_{i},\delta _{i})\right\rangle }\delta _{i}^{t\left\langle q,\alpha
\right\rangle }.
\end{equation*}%
It results that 
\begin{equation}
\underset{i\in I}{\sum }r_{i}^{\gamma \lambda }\leq k\overline{P}_{\chi
,u_{\varepsilon }}^{(\gamma -t)q,t\left\langle q,\alpha \right\rangle
}(A_{m}).  \label{eqp2v}
\end{equation}%
We note that from the inequality (\ref{eqp1v}), there exists $\varepsilon
_{1}>0$ such that for $u_{\varepsilon }<\varepsilon _{1},$ 
\begin{equation*}
\overline{P}_{\chi ,u_{\varepsilon }}^{(\gamma -t)q,t\left\langle q,\alpha
\right\rangle }(A_{m})<+\infty .
\end{equation*}%
Then from the inequality (\ref{eqp2v}) it comes that for all $m$ $\in M$, 
\begin{equation*}
\triangle (A_{m})\leq \gamma \lambda .
\end{equation*}%
Therefore 
\begin{equation*}
\Dim(A_{m})\leq \gamma \lambda ,\text{ }m\in M.
\end{equation*}%
From the equality (\ref{eqp0v}), we write 
\begin{equation*}
\Dim(A^{\left\langle q,\alpha \right\rangle })\leq \gamma \lambda ,\text{ }%
m\in M.
\end{equation*}%
Finally, for all $t<0,$ we obtain 
\begin{equation*}
\Dim(A^{\left\langle q,\alpha \right\rangle })\leq \Phi
_{q}(t)L^{q}(A^{\left\langle q,\alpha \right\rangle }).
\end{equation*}
\end{proof}

Thereafter, let $\alpha \in \mathbb{E}^{\prime }$ and $q\in E.$\newline
We set 
\begin{equation*}
\Phi _{q}=\underset{t<0}{\inf }\left( \Phi _{q}(t)\right) .
\end{equation*}%
and 
\begin{equation*}
X_{\chi }^{q}\left( \alpha \right) =\left\{ x\in \mathbb{X};\text{\thinspace 
}\underset{r\rightarrow 0}{\lim \sup }\frac{\left\langle q,\chi
(x,r)\right\rangle }{\log r}\leq \left\langle q,\alpha \right\rangle
\right\} .
\end{equation*}%
For all real number $\eta >0$ and $p\geq 1$ an integer, we set 
\begin{equation*}
X_{\left\langle q,\alpha \right\rangle }(\eta ,p)=\left\{ x\in X_{\chi
}\left( \alpha \right) ;\text{ }r^{\left\langle q,\alpha \right\rangle +\eta
}\leq e^{\left\langle q,\chi (x,r)\right\rangle }\text{ for }r<\frac{1}{p}%
\right\}
\end{equation*}%
and 
\begin{equation*}
T_{\chi }^{q}(\alpha ,\eta ,p)=\underset{A\subset X_{\left\langle q,\alpha
\right\rangle }(\eta ,p)}{\sup }L^{q}(A)\text{,}
\end{equation*}%
\begin{equation*}
T_{\chi }^{q}(\alpha ,\eta )=\underset{p\rightarrow +\infty }{\lim }T_{\chi
}^{q}(\alpha ,\eta ,p),
\end{equation*}%
\begin{equation*}
T_{\chi }^{q}(\alpha )=\underset{\eta \rightarrow 0^{+}}{\lim }T_{\chi
}^{q}(\alpha ,\eta ).
\end{equation*}

\begin{theorem}
\label{th3v} 
\begin{equation*}
\Dim(X_{\chi }\left( \alpha ,E\right) )\leq \underset{q\in E}{\inf }\left\{
\Phi _{q}T_{\chi }^{q}(\alpha )\right\} .
\end{equation*}
\end{theorem}

\begin{proof}
Let $q\in E$ and suppose that $X_{\chi }^{q}\left( \alpha \right) \neq
\varnothing $. From the theorem \ref{thIv} it comes that for all $\eta >0,$%
\begin{equation*}
\Dim(X_{\left\langle q,\alpha \right\rangle }(\eta ,p))\leq \Phi
_{q}(t)L^{q}(X_{\left\langle q,\alpha \right\rangle }(\eta ,p)).
\end{equation*}%
Thus 
\begin{equation*}
\Dim(X_{\left\langle q,\alpha \right\rangle }(\eta ,p))\leq \Phi
_{q}(t)T_{\chi }^{q}(\alpha ,\eta ,p).
\end{equation*}%
Then for all $\eta >0,$ 
\begin{equation*}
\Dim(\underset{p}{\cup }X_{\left\langle q,\alpha \right\rangle }(\eta
,p))\leq \Phi _{q}(t)T_{\chi }^{q}(\alpha ,\eta ).
\end{equation*}%
We remark that for all $\eta >0,$%
\begin{equation*}
X_{\chi }^{q}\left( \alpha \right) \subset \underset{p}{\cup }%
X_{\left\langle q,\alpha \right\rangle }\left( \eta ,p\right) .
\end{equation*}%
It results that for all $\eta >0,$ 
\begin{equation*}
\Dim(X_{\chi }^{q}\left( \alpha \right) )\leq \Phi _{q}(t)T_{\chi
}^{q}(\alpha ,\eta ).
\end{equation*}%
Letting $\eta \rightarrow 0$, we obtain that 
\begin{equation*}
\Dim(X_{\chi }^{q}\left( \alpha \right) )\leq \Phi _{q}(t)T_{\chi
}^{q}(\alpha ).
\end{equation*}%
So it comes that 
\begin{equation*}
\Dim(X_{\chi }^{q}\left( \alpha \right) )\leq \underset{t<0}{\inf }\left(
\Phi _{q}(t)\right) T_{\chi }^{q}(\alpha ).
\end{equation*}%
It is clear that $X_{\chi }\left( \alpha ,E\right) =\underset{q\in E}{\cap }%
X_{\chi }^{q}\left( \alpha \right) .$ Then for all $q\in E$%
\begin{equation*}
\Dim(X_{\chi }\left( \alpha ,E\right) )\leq \Phi _{q}T_{\chi }^{q}(\alpha ).
\end{equation*}%
Finally it follows that 
\begin{equation*}
\Dim(X_{\chi }\left( \alpha ,E\right) )\leq \underset{q\in E}{\inf }\left\{
\Phi _{q}T_{\chi }^{q}(\alpha )\right\} .
\end{equation*}
\end{proof}

We have just established an other upper bounding for $\Dim(X_{\chi }\left(
\alpha ,E\right) )$ which is in some cases thinner than that established by
J. Peyrière in \cite{J.P2} as shown in the example below.

\section{Example}

To build the example, we first define the metric space $\mathbb{X}$ and then
choose the function $\chi $.

We denote by $\mathcal{A}$ the set $\left\{ 0,1\right\} $ and by $\mathcal{A}%
^{n}$ all words of length $n$ constructed with $\mathcal{A}$ as alphabet.
The empty word is denoted by $\epsilon $. For all $j\in \mathcal{A}^{n}$, we
set $N_{0}(j)$ the number of occurrence of the letter $0$ in the word $j$.

Let $j$ and $j^{\prime }$ two words, we denote by $jj^{\prime }$ the
concatenation of $j$ and $j^{\prime }$.

We denote by $\mathbb{X}$ the symbolic space $\left\{ 0,1\right\} ^{\mathbb{N%
}}$, ie the set of sequences $(x_{i})_{i\geq 0}$ of elements of $\{0,1\}$.
Is defined in the same way the concatenation of a finite word and an
infinite word.

If $x=(x_{i})_{i\geq 0}$, $y=(y_{i})_{i\geq 0}\in \mathbb{X}$, we set 
\begin{equation*}
d(x,y)=\left\{ 
\begin{array}{l}
0,\text{ \ \ \ if }x=y, \\ 
2^{-n},\text{ if }x_{n}\neq y_{n}\text{ and }x_{i}=y_{i}\text{ for all }%
0\leq i<n.%
\end{array}%
\right.
\end{equation*}%
If $j=j_{0}j_{1}...j_{n-1}\in \mathcal{A}^{n}$, we set the cylinder 
\begin{equation*}
\left[ j\right] =\left[ j_{0}j_{1}...j_{n-1}\right] =\left\{ jx,\text{ }x\in 
\mathbb{X}\right\} .
\end{equation*}%
It is clear that if $x=(x_{i})_{i\geq 0}\in \mathbb{X}$ and $2^{-n-1}\leq
r<2^{-n},$ then 
\begin{equation*}
B(x,r)=\left[ x_{0}x_{1}...x_{n}\right] .
\end{equation*}

Let $\mathcal{L}$ be a family of cylinders, any element $\left[ j\right] $
of $\mathcal{L}$ is called selected cylinder.

Let $0<p_{0}\leq p_{1}$ such that $p_{0}+p_{1}=1$.

We associate the measure $\mu $ on $\mathbb{X}$ such that for any cylinder $%
\left[ j\right] $ and $l\in \left\{ 0,1\right\} $, 
\begin{equation*}
\mu \left( \left[ jl\right] \right) =\left\{ 
\begin{array}{l}
p_{l}\text{ }\mu \left( \left[ j\right] \right) \text{, if }\left[ j\right] 
\text{ contains a selected cylinder,} \\ 
\dfrac{\mu \left( \left[ j\right] \right) }{2}\text{, otherwise.}%
\end{array}%
\right.
\end{equation*}

For the construction of the example we choose the part $\mathcal{L}$ as
follows.

Let $\beta _{1},$ $\beta _{2},$ $\gamma _{1}$ and $\gamma _{2}$ be real
numbers such that 
\begin{equation*}
\frac{1}{2}<\beta _{1}<\gamma _{1}<\beta _{2}<\gamma _{2}<\frac{1}{3}.
\end{equation*}%
We say that the cylinder $\left[ j\right] $ such that $j\in \mathcal{A}^{n}$
is of 
\begin{equation*}
\begin{array}{c}
\text{type }T_{1},\text{ if \ }\beta _{1}<\dfrac{N_{0}(j)}{n}<\gamma _{1},
\\ 
\text{type }T_{2},\text{ if \ }\beta _{2}<\dfrac{N_{0}(j)}{n}<\gamma _{2}.%
\end{array}%
\end{equation*}%
Let $j\in \mathcal{A}^{n}$ such that $\left[ j\right] $ is of type $1$
(respectively of type $2$), put $\widetilde{\left[ j\right] }$ the set of
the cylinders $\left[ j^{\prime }\right] $, $j^{\prime }\in \mathcal{A}%
^{n+6},$ contained in $\left[ j\right] $ and of the same type than $\left[ j%
\right] $.\newline
Let $n_{0}\in \mathbb{N}$ be a multiple of $6$ and $\left( n_{p}\right) $
the sequence of integers defined by 
\begin{equation*}
n_{0},\text{ }n_{3i+1}=2^{n_{3i}}n_{0},\text{ }n_{3i+2}=2n_{3i+1}\text{ and }%
n_{3i+3}=2n_{3i+2}.
\end{equation*}%
For $k\in \mathbb{N}$ we construct the family $\mathcal{G}_{k}$ of disjoint
cylinders $\left[ j\right] ,$ $j\in \mathcal{A}^{n_{0}+6k}$ such that :%
\newline
$\circ $ any element $\left[ j\right] $ of $\mathcal{G}_{k}$ such that $j\in 
\mathcal{A}^{n}$ satisfies the relation 
\begin{equation*}
\beta _{1}<\frac{N_{0}(j)}{n}<\gamma _{2},
\end{equation*}%
\medskip $\circ $ $\mathcal{G}_{0}$ contains two cylinders $\left[ j^{1}%
\right] $ and $\left[ j^{2}\right] $ respectively of type $T_{1}$ and $T_{2}$%
,\newline
$\circ $ any element of $\mathcal{G}_{k+1}$ is contained in an element of $%
\mathcal{G}_{k}$ called his father,\newline
$\circ $ all elements of $\mathcal{G}_{k}$ beget the same number of son in $%
\mathcal{G}_{k+1}$, and from the generation $\mathcal{G}_{k}$ to generation $%
\mathcal{G}_{k+1}$ we distinguish the following three cases: \newline
\textbf{1}$^{\text{\textbf{st}}}$\textbf{\ case:} If $n_{3i}\leq
n_{0}+6k<n_{3i+1},$ then for all $\left[ j\right] \in \mathcal{G}_{k}$ we
select two cylinders in $\widetilde{\left[ j\right] }.$ Then $\mathcal{G}%
_{k+1}$ is the union of all these selected cylinders. \newline
\textbf{2}$^{\text{\textbf{nd}}}$\textbf{\ case:} If $n_{3i+1}\leq
n_{0}+6k<n_{3i+2},$ then for all $\left[ j\right] $ $\in \mathcal{G}_{k}$ of
type $T_{1}$, we select a cylinder in $\widetilde{\left[ j\right] }$ and for
all $\left[ j\right] \in \mathcal{G}_{k}$ of type $T_{2}$ we select a
cylinder $\left[ j^{\prime }\right] $, $j^{\prime }\in \mathcal{A}%
^{n_{0}+6(k+1)}$ containing a cylinder selected in $\mathcal{G}_{n_{3i+2}}$
of type $T_{1}$. Then $\mathcal{G}_{k+1}$ is the union of all these selected
cylinders.\newline
Note that all cylinders in $\mathcal{G}_{n_{3i+2}}$ are of type $T_{1}$.%
\newline
\textbf{3}$^{\text{\textbf{rd}}}$\textbf{\ case:} If $n_{3i+2}\leq
n_{0}+6k<n_{3i+3},$ then for all $\left[ j\right] $ $\in \mathcal{G}_{k}$
having an ancestor in $\mathcal{G}_{n_{3i+1}}$ of type $T_{1}$, we select a
cylinder $\left[ j^{\prime }\right] $, $j^{\prime }\in \mathcal{A}%
^{n_{0}+6(k+1)}$, containing a selected cylinder in $\mathcal{G}_{n_{3i+3}}$
of type $T_{2},$and for all $\left[ j\right] $ $\in \mathcal{G}_{k}$ of type 
$T_{1}$ we select a cylinder in $\widetilde{\left[ j\right] }$. Then $%
\mathcal{G}_{k+1}$ is the union of all these selected cylinders.

for $n_{0}$ large enough, this construction is possible and we can impose
the following separation condition :

For all $\left[ j\right] ,$ $\left[ j^{\prime }\right] \in \mathcal{G}_{k}$
such that $j,$ $j^{\prime }\in \mathcal{A}^{n}$, for all $k\geq 0,$ the
distance between $\left[ j\right] $ and $\left[ j^{\prime }\right] $ is
larger than $\dfrac{1}{2^{n-2}}$ and for all $k\geq 1$, the distance between 
$\left[ j\right] $ and an element of his father is larger than $\dfrac{1}{%
2^{n-1}}$.

We choose $\mathcal{L}=\left( \underset{k\geq 0}{\cup }\mathcal{G}%
_{k}\right) $ and we associate the following relation on $\mathcal{L}$ :

the two elements of $\mathcal{G}_{0}$ are related and two element of $%
\mathcal{G}_{k+1}$ are related if their fathers elements of $\mathcal{G}_{k}$%
, are related.

Now put 
\begin{equation*}
E=\left\{ q=(q_{1},q_{2})\in \mathbb{R}^{2};\text{ }q_{1}+q_{2}\geq
0\right\} .
\end{equation*}

The function $\chi :\mathbb{X\times }\left] 0,1\right] \rightarrow \mathbb{E}%
^{\prime }$ is defined such that for all $q=(q_{1},q_{2})\in \mathbb{R}^{2}$
and for all $\lambda >0,$ there exists $r_{0}>0$ such that for $x\in \mathbb{%
X}$ and $r<r_{0}$ 
\begin{equation*}
r^{\lambda }\mu (B(x,r))^{(q_{1}+q_{2})}\leq e^{\left\langle q,\chi
(x,r)\right\rangle }\leq r^{-\lambda }\mu (B(x,r))^{(q_{1}+q_{2})}.
\end{equation*}%
Let $a>0$ , for all $q=(q_{1},q_{2})\in \mathbb{R}^{2}$ we set 
\begin{equation*}
\left\langle q,\alpha \right\rangle =a(q_{1}+q_{2}).
\end{equation*}%
We denote 
\begin{equation*}
\overline{X}^{a}=\left\{ x\in \mathbb{X},\text{\thinspace }\underset{%
r\rightarrow 0}{\lim \sup }\frac{\log \mu (B(x,r))}{\log r}\leq a\right\} .
\end{equation*}

\begin{proposition}
\label{thv1sym1} 
\begin{equation*}
X_{\chi }\left( \alpha ,E\right) =\overline{X}^{a}.
\end{equation*}
\end{proposition}

\begin{proof}
From the inequalities 
\begin{equation*}
r^{\lambda }\mu (B(x,r))^{(q_{1}+q_{2})}\leq e^{\left\langle q,\chi
(x,r)\right\rangle }\leq r^{-\lambda }\mu (B(x,r))^{(q_{1}+q_{2})}
\end{equation*}%
we deduce that 
\begin{equation*}
-\lambda +(q_{1}+q_{2})\frac{\log (\mu (B(x,r))}{\log r}\leq \frac{%
\left\langle q,\chi (x,r)\right\rangle }{\log r}\leq \lambda +(q_{1}+q_{2})%
\frac{\log (\mu (B(x,r))}{\log r}.
\end{equation*}%
It follows that 
\begin{equation*}
X_{\chi }\left( \alpha ,E\right) =\left\{ x\in \mathbb{X},\text{\thinspace }%
(q_{1}+q_{2})\underset{r\rightarrow 0}{\lim \sup }\frac{\log \mu (B(x,r))}{%
\log r}\leq a(q_{1}+q_{2})\right\} .
\end{equation*}%
Let us recall that for all $q=(q_{1},q_{2})\in E;$ $q_{1}+q_{2}\geq 0$, then
it is easy to obtain that%
\begin{equation*}
X_{\chi }\left( \alpha ,E\right) =\left\{ x\in \mathbb{X},\text{\thinspace }%
\underset{r\rightarrow 0}{\lim \sup }\frac{\log \mu (B(x,r))}{\log r}\leq
a\right\}
\end{equation*}%
or 
\begin{equation*}
X_{\chi }\left( \alpha ,E\right) =\overline{X}^{a}.
\end{equation*}
\end{proof}

We write for all real number $\theta $, 
\begin{equation*}
\Lambda _{\mu }(\theta )=\Delta _{\mu }^{\theta }(\supp\mu )\text{ \ and \ }%
B_{\mu }(\theta )=\Dim_{\mu }^{\theta }(\supp\mu ),
\end{equation*}%
Let us recall that it is already established in \cite{O.L1} that 
\begin{equation*}
B_{\mu }\leq \Lambda _{\mu }
\end{equation*}%
\begin{equation*}
B_{\mu }(1)=\Lambda _{\mu }(1)=0
\end{equation*}%
and that the functions $\Lambda _{\mu }:\theta \mapsto \Lambda _{\mu
}(\theta )$ and $B_{\mu }:\theta \mapsto B_{\mu }(\theta )$ are convex and
decreasing.

\begin{proposition}
For $q=(q_{1},q_{2})\in \mathbb{R}^{2},$ 
\begin{equation*}
\Lambda _{\chi }(q)=\Lambda _{\mu }(q_{1}+q_{2})
\end{equation*}%
and%
\begin{equation*}
B_{\chi }(q)=B_{\mu }(q_{1}+q_{2}).
\end{equation*}
\end{proposition}

\begin{proof}
It suffices to note that for $t\in \mathbb{R},$ $q=(q_{1},q_{2})\in \mathbb{R%
}^{2}$ and for $\lambda >0,$ there exists $r_{0}>0$ such that for $%
r_{i}<r_{0}$ and $(B\left( x_{i},r_{i}\right) )_{i\in I}$ a centered $%
\varepsilon -$packing of $\mathbb{X}$ with $\varepsilon \leq r_{0},$ 
\begin{equation*}
\underset{i}{\sum }r_{i}^{t+\lambda }\mu
(B(x_{i},r_{i}))^{(q_{1}+q_{2})}\leq \underset{i}{\sum }r_{i}^{t}e^{\left%
\langle q,\chi (x_{i},r_{i})\right\rangle }\leq \underset{i}{\sum }%
r_{i}^{t-\lambda }\mu (B(x_{i},r_{i}))^{(q_{1}+q_{2})}.
\end{equation*}%
Then%
\begin{equation*}
\overline{P}_{\mu ,\varepsilon }^{(q_{1}+q_{2}),t+\lambda }(\mathbb{X})\leq 
\overline{P}_{\chi ,\varepsilon }^{q,t}(\mathbb{X})\leq \overline{P}_{\mu
,\varepsilon }^{(q_{1}+q_{2}),t-\lambda }(\mathbb{X}),
\end{equation*}%
letting $\varepsilon \rightarrow 0$, it comes that%
\begin{equation*}
\overline{P}_{\mu }^{(q_{1}+q_{2}),t+\lambda }(\mathbb{X})\leq \overline{P}%
_{\chi }^{q,t}(\mathbb{X})\leq \overline{P}_{\mu }^{(q_{1}+q_{2}),t-\lambda
}(\mathbb{X}),
\end{equation*}%
and letting $\lambda \rightarrow 0$, we obtain the equality 
\begin{equation*}
\overline{P}_{\chi }^{q,t}(\mathbb{X})=\overline{P}_{\mu }^{(q_{1}+q_{2}),t}(%
\mathbb{X}).
\end{equation*}%
Then it is clear that 
\begin{equation*}
\Lambda _{\chi }(q)=\Lambda _{\mu }(q_{1}+q_{2})
\end{equation*}%
and 
\begin{equation*}
B_{\chi }(q)=B_{\mu }(q_{1}+q_{2}).
\end{equation*}
\end{proof}

\begin{proposition}
\label{thv1sym}%
\begin{equation*}
\underset{q\in E}{\inf }(\left\langle q,\alpha \right\rangle +B_{\chi }(q))=%
\underset{\theta \geq 0}{\inf }(a\theta +B_{\mu }(\theta )).
\end{equation*}
\end{proposition}

\begin{proof}
As $B_{\chi }(q)=B_{\mu }(q_{1}+q_{2})$ it is clear that 
\begin{equation*}
\underset{q\in E}{\inf }(\left\langle q,\alpha \right\rangle +B_{\chi }(q))=%
\underset{q_{1}+q_{2}\geq 0}{\inf }(a(q_{1}+q_{2})+B_{\mu }(q_{1}+q_{2}))
\end{equation*}

or%
\begin{equation*}
\underset{q\in E}{\inf }(\left\langle q,\alpha \right\rangle +B_{\chi }(q))=%
\underset{\theta \geq 0}{\inf }(a\theta +B_{\mu }(\theta )).
\end{equation*}
\end{proof}

We find in the following corollary a theorem obtained by L. Olsen in \cite%
{O.L1}.

\begin{corollary}
\begin{equation*}
\Dim(\overline{X}^{a})\leq \underset{\theta \geq 0}{\inf }(a\theta +B_{\mu
}(\theta )).
\end{equation*}
\end{corollary}

\begin{proof}
By applying the theorem \ref{th1v} established by J. Peyrière we get that 
\begin{equation*}
\Dim(X_{\chi }\left( \alpha ,E\right) )\leq \underset{q\in E}{\inf }%
(\left\langle q,\alpha \right\rangle +B_{\chi }(q))
\end{equation*}%
which gives using the propositions \ref{thv1sym1} and \ref{thv1sym} that 
\begin{equation*}
\Dim(\overline{X}^{a})\leq \underset{\theta \geq 0}{\inf }(a\theta +B_{\mu
}(\theta )).
\end{equation*}
\end{proof}

\begin{proposition}
\label{prop4sym} 
\begin{equation*}
\underset{\theta \rightarrow +\infty }{\lim }B_{\mu }(\theta )=-\infty .
\end{equation*}
\end{proposition}

\begin{proof}
We note that for all $\left[ j\right] $ such that $j\in \mathcal{A}^{n}$, 
\begin{equation}
p_{0}^{n}\leq \mu \left( \left[ j\right] \right) \leq p_{1}^{n}.  \label{eq7}
\end{equation}%
Let $(B(x_{i},r_{i}))_{i\in I}$ be a centered $\varepsilon $ -packing of $%
\mathbb{X}$.\newline
For all $i\in I,$ we consider the cylinder $\left[ j\right]
_{i}=B(x_{i},r_{i})$ such that $j\in \mathcal{A}^{n+1}$ and 
\begin{equation}
\frac{1}{2^{n+1}}\leq r_{i}<\frac{1}{2^{n}}.  \label{eq8}
\end{equation}%
Given (\ref{eq7}), we get that 
\begin{equation}
p_{0}^{n}\leq \mu (B(x_{i},r_{i}))\leq p_{1}^{n}.  \label{eq9}
\end{equation}%
From (\ref{eq8}), we deduce that for all $t,$ there exist $c_{1}$ and $c_{2} 
$ such that for all $n\in \mathbb{N}$, 
\begin{equation}
\frac{c_{1}}{2^{nt}}\leq r_{i}{}^{t}\leq \frac{c_{2}}{2^{nt}},  \label{eq10}
\end{equation}%
and from (\ref{eq9}), it comes that for all $\theta >0$, 
\begin{equation}
p_{0}^{n\theta }\leq \mu \left( B\left( x_{i},r_{i}\right) \right) ^{\theta
}\leq p_{1}^{n\theta }.  \label{eq11}
\end{equation}%
Then given (\ref{eq10}) and (\ref{eq11}), there exists $c_{3}$ such that 
\begin{equation}
\mu \left( B\left( x_{i},r_{i}\right) \right) ^{\theta }r_{i}{}^{t}\leq
c_{3}p_{1}^{n\theta }2^{-nt}.  \label{eq12}
\end{equation}%
It follows thanks to (\ref{eq12}), that there exists $C$ which depends only
on $\theta $ and $t$ such that 
\begin{equation}
\underset{\frac{1}{2^{n}}\leq 2r_{i}\leq \frac{1}{2^{n-2}}}{\sum }\mu \left(
B\left( x_{i},r_{i}\right) \right) ^{\theta }r_{i}{}^{t}\leq C(p_{1}^{\theta
}2^{-t})^{n}.  \label{eq13}
\end{equation}%
Writing for $\varepsilon >0$ small enough, 
\begin{equation*}
\underset{i\in I}{\sum }\mu \left( B\left( x_{i},r_{i}\right) \right)
^{\theta }r_{i}{}^{t}=\underset{n\geq 0}{\sum }\underset{\frac{1}{2^{n+1}}%
\leq r_{i}<\frac{1}{2^{n}}}{\sum }\mu \left( B\left( x_{i},r_{i}\right)
\right) ^{\theta }r_{i}{}^{t},
\end{equation*}%
it follows from inequality (\ref{eq13}) that 
\begin{equation*}
\underset{i\in I}{\sum }\mu \left( B\left( x_{i},r_{i}\right) \right)
^{\theta }r_{i}{}^{t}<+\infty ,\text{ }t>\theta \dfrac{\log p_{1}}{\log 2}.
\end{equation*}%
We deduce that%
\begin{equation*}
\Lambda _{\mu }(\theta )\leq \theta \frac{\log p_{1}}{\log 2},\text{ }\theta
>0.
\end{equation*}%
Then 
\begin{equation*}
B_{\mu }(\theta )\leq \theta \frac{\log p_{1}}{\log 2},\theta >0.
\end{equation*}%
Finally we obtain that 
\begin{equation*}
\underset{\theta \rightarrow +\infty }{\lim }B_{\mu }(\theta )=-\infty .
\end{equation*}
\end{proof}

\begin{proposition}
\label{prop5symbis}Put $B_{\mu -}^{^{\prime }}(1)$ the left derivative
number of $B_{\mu }$ at $1$. Then 
\begin{equation*}
B_{\mu -}^{^{\prime }}(1)\leq -1.
\end{equation*}
\end{proposition}

\begin{proof}
As $B_{\mu }(1)=0$ and $B_{\mu }$ is convex, it comes that to prove that $%
B_{\mu -}^{^{\prime }}(1)\leq -1$, it is sufficient to establish that for
all $\theta <1,$ 
\begin{equation*}
B_{\mu }(\theta )\geq 1-\theta ,
\end{equation*}%
which amounts given (\ref{eq2v}), to show that if $\left( \underset{i}{\cup }%
E_{i}\right) $ is a partition of $\mathbb{X}$, then $\underset{i\in I}{\sum }%
\overline{P}_{\mu }^{\theta ,t}(E_{i})=+\infty $. \newline
Let us consider the case where for all $i\in I,$ $\overline{P}_{\mu
}^{\theta ,t}(E_{i})<+\infty $.\newline
Let $0<\varepsilon <\frac{1}{2^{n_{0}}}.$ For all $i\in I$, we choose $%
\delta _{i}<\varepsilon $ such that 
\begin{equation}
\overline{P}_{\mu ,\delta _{i}}^{\theta ,t}(E_{i})\leq \overline{P}_{\mu
}^{\theta ,t}(E_{i})+\frac{1}{2^{i}}.  \label{eq14}
\end{equation}%
As the space $\mathbb{X}$ satisfies the Besicovitch covering property, there
exists an integer $\zeta $ (which depends only on $\mathbb{X}$) such that
each $E_{i}$ is covered by $\overset{\zeta }{\underset{u=1}{\cup }}\left( 
\underset{j}{\cup }B\left( x_{ij},\delta _{i}\right) \right) _{u}$ such that
for all $1\leq u\leq \zeta $, $\left( B\left( x_{ij},\delta _{i}\right)
\right) _{j}$ is a packing. \newline
Given (\ref{eq14}), it comes that 
\begin{equation*}
\underset{u=1}{\overset{\zeta }{\sum }}\underset{j}{\sum }\mu \left( B\left(
x_{ij},\delta _{i}\right) \right) ^{\theta }\delta _{i}^{t}\leq \zeta \left( 
\overline{P}_{\mu }^{\theta ,t}(E_{i})+\frac{1}{2^{i}}\right) .
\end{equation*}%
Then, 
\begin{equation}
\underset{i}{\sum }\left( \underset{u=1}{\overset{\zeta }{\sum }}\underset{j}%
{\sum }\mu \left( B\left( x_{ij},\delta _{i}\right) \right) ^{\theta }\delta
_{i}^{t}\right) \leq \zeta \underset{i}{\sum }\overline{P}_{\mu }^{\theta
,t}(E_{i})+\zeta .  \label{eq15}
\end{equation}%
let us consider the sum 
\begin{equation}
\underset{i}{\sum }\left( \underset{u=1}{\overset{\zeta }{\sum }}\underset{l}%
{\sum }^{^{\prime }}\mu \left( B\left( x_{il},\delta _{i}\right) \right)
^{\theta }\delta _{i}^{t}\right) ,  \label{eq16}
\end{equation}%
where $\underset{l}{\sum }^{^{\prime }}$ is taken over all $l$ such that the
distance between $x_{il}$ and $\left[ j^{1}\right] $ (respectively $\left[
j^{2}\right] $) is larger than $\dfrac{1}{2^{n_{0}-1}}.$\newline
In this case, there exists $c$ which depends only on $n_{0}$ such that 
\begin{equation*}
\mu \left( B\left( x_{il},\delta _{i}\right) \right) \leq c\,m\left( B\left(
x_{il},\delta _{i}\right) \right) ,
\end{equation*}%
where $m$ is the Lebesgue measure.\newline
Then there exists $C$ such that 
\begin{equation}
C^{\theta -1}\delta _{i}^{\theta -1+t}\leq \mu \left( B\left( x_{il},\delta
_{i}\right) \right) ^{\theta -1}\delta _{i}^{t}.  \label{eq17}
\end{equation}%
Moreover, the union of the balls contained in the sum (\ref{eq16}) covers $%
\mathbb{X}$ deprived of at most 6 cylinders of the generation $n_{0}$,
therefore given (\ref{eq17}), we obtain that 
\begin{equation*}
\left( 1-\frac{6}{2^{n_{0}-1}}\right) C^{\theta -1}\varepsilon ^{\theta
-1+t}\leq \underset{i}{\sum }\left( \underset{u=1}{\overset{\zeta }{\sum }}%
\underset{l}{\sum }^{^{\prime }}\mu \left( B\left( x_{il},\delta _{i}\right)
\right) ^{\theta }\delta _{i}^{t}\right) .
\end{equation*}%
We deduce according to (\ref{eq15}), 
\begin{equation*}
\left( 1-\frac{6}{2^{n_{0}}}\right) C^{\theta -1}\varepsilon ^{\theta
-1+t}\leq \zeta \underset{i}{\sum }\overline{P}_{\mu }^{\theta
,t}(E_{i})+\zeta .
\end{equation*}%
Letting $\varepsilon \rightarrow 0,$ it results that $\underset{i\in I}{\sum 
}\overline{P}_{\mu }^{\theta ,t}(E_{i})=+\infty $ while $t<1-\theta ,$ 
\newline
so that 
\begin{equation*}
B_{\mu }(\theta )\geq 1-\theta ,\text{ }\theta <1.
\end{equation*}
\end{proof}

\begin{proposition}
\label{prop6sym} We set 
\begin{equation*}
\mathcal{C}=\underset{k\geq 1}{\cap }\left( \underset{\left[ j\right] \in 
\mathcal{G}_{k}}{\cup }\left[ j\right] \right)
\end{equation*}%
and the function $g$ defined on $\left[ 0,1\right] $ by 
\begin{equation*}
g(x)=-\frac{x\log \left( \frac{p_{0}}{p_{1}}\right) +\log p_{1}}{\log 2}.
\end{equation*}%
i. If $x\notin \mathcal{C},$ then 
\begin{equation*}
\underset{r\rightarrow 0}{\lim }\frac{\log (\mu \left( B(x,r)\right) }{\log r%
}=1.
\end{equation*}%
ii. If $x\in \mathcal{C},$ then 
\begin{equation*}
g(\beta _{1})\leq \underset{r\rightarrow 0}{\lim \inf }\frac{\log (\mu
\left( B(x,r)\right) }{\log r}\leq \underset{r\rightarrow 0}{\lim \sup }%
\frac{\log (\mu \left( B(x,r)\right) }{\log r}\leq g(\gamma _{2}).
\end{equation*}
\end{proposition}

\begin{proof}
\textit{i}. Let $x\notin \mathcal{C}.$ Thanks to the separation condition,
for $r>0$ small enough, the ball $B(x,r)=\left[ j\right] $ such that $j\in 
\mathcal{A}^{n+1}$, $\dfrac{1}{2^{n+1}}\leq r<\dfrac{1}{2^{n}}$ and $\left[ j%
\right] $ do not meet $\mathcal{C}$. There exists $c$ such that 
\begin{equation*}
\mu \left( \left[ j\right] \right) =\frac{c}{2^{n}}.
\end{equation*}%
We deduce that 
\begin{equation*}
\underset{n\rightarrow +\infty }{\lim }\frac{\log (\mu \left( \left[ j\right]
\right) }{\log \left( \dfrac{1}{2^{n+1}}\right) }=1
\end{equation*}%
i.e 
\begin{equation*}
\underset{r\rightarrow 0}{\lim }\frac{\log (\mu \left( B(x,r)\right) }{\log r%
}=1.
\end{equation*}

\textit{ii}. It is clear that if $\left[ j\right] \in \mathcal{G}_{k}$ such
that $j\in \mathcal{A}^{n}$, then 
\begin{equation*}
\mu \left( \left[ j\right] \right) =p_{0}^{N_{0}(j)}p_{1}^{n-N_{0}(j)},
\end{equation*}%
so that 
\begin{equation*}
\mu (\left[ j\right] )=\left( \dfrac{1}{2^{n}}\right) ^{g\left( \frac{%
N_{0}(j)}{n}\right) }.
\end{equation*}%
Furthermore, we recall that, 
\begin{equation*}
\beta _{1}<\dfrac{N_{0}(j)}{n}<\gamma _{2}.
\end{equation*}%
The function $g$ is strictly increasing, it comes that 
\begin{equation}
g\left( \beta _{1}\right) <\dfrac{\log (\mu \left[ j\right] )}{\log \left( 
\dfrac{1}{2^{n}}\right) }<g\left( \gamma _{2}\right) .  \label{eq20}
\end{equation}%
Let $x\in \mathcal{C}$ and $r<\dfrac{1}{2^{n_{0}+6}}$, then $B(x,r)$ is
contained in one of the selected cylinders $\left[ j^{1}\right] $ or $\left[
j^{2}\right] .$\newline
We consider the selected cylinder $\left[ j\right] $ such that $j\in 
\mathcal{A}^{n+1}$, $\dfrac{1}{2^{n+1}}\leq r<\dfrac{1}{2^{n}}$ and $\left[ j%
\right] =B(x,r)$, we can write 
\begin{equation*}
\mu \left( B(x,r)\right) =\mu \left( \left[ j\right] \right) .
\end{equation*}%
Given (\ref{eq20}), it results that 
\begin{equation*}
g(\beta _{1})\leq \underset{r\rightarrow 0}{\lim \inf }\frac{\log (\mu
\left( B(x,r)\right) }{\log r}\leq \underset{r\rightarrow 0}{\lim \sup }%
\frac{\log (\mu \left( B(x,r)\right) }{\log r}\leq g(\gamma _{2}).
\end{equation*}
\end{proof}

Subsequently, even if we choose $p_{0}>\gamma _{2}$, we stand in the case
where $g(\gamma _{2})<1.$

Thus under the proposition \ref{prop6sym}, 
\begin{equation*}
\overline{X}^{g(\gamma _{2})}=\mathcal{C}.
\end{equation*}

Let $a>0$ such that $g(\gamma _{1})<a\leq g(\gamma _{2})$ and\ $\overline{X}%
^{a}\neq \varnothing $.

\begin{proposition}
\label{prop7sym} 
\begin{equation*}
\underset{q\in E}{\inf }\left\{ \Phi _{q}T_{\chi }^{q}(\alpha )\right\} <%
\underset{q\in E}{\inf }(\left\langle q,\alpha \right\rangle +B_{\chi }(q)).
\end{equation*}
\end{proposition}

\begin{proof}
Put $s=\underset{\theta }{\inf }B_{\mu }(\theta )<0.$ It is clear that 
\begin{equation*}
\underset{q\in E}{\inf }\left\{ \Phi _{q}T_{\chi }^{q}(\alpha )\right\} \leq 
\underset{q\in E}{\inf }\left\{ T_{\chi }^{q}(\alpha )\underset{s<t<0}{\inf }%
\left( \Phi _{q}(t)\right) \right\} .
\end{equation*}%
We put 
\begin{equation*}
F=\left\{ q=(q_{1},q_{2})\in E;\text{ }q_{1}+q_{2}=1\right\} .
\end{equation*}%
It follows that for all $q\in F$, $T_{\chi }^{q}(\alpha )$ keeps a constant
value which we denote by $T_{\chi }^{F}(\alpha )$, we can write%
\begin{equation*}
\underset{q\in E}{\inf }\left\{ \Phi _{q}T_{\chi }^{q}(\alpha )\right\} \leq
T_{\chi }^{F}(\alpha )\underset{s<t<0}{\inf }\left( \Phi _{q}(t)\right) .
\end{equation*}%
Also the equality 
\begin{equation*}
\Phi _{q}(t)=\inf \left\{ \gamma >0;\text{ }ta(q_{1}+q_{2})>B_{\mu }((\gamma
-t)(q_{1}+q_{2}))\right\}
\end{equation*}%
gives 
\begin{equation*}
\underset{q\in E}{\inf }\left\{ \Phi _{q}T_{\chi }^{q}(\alpha )\right\} \leq
T_{\chi }^{F}(\alpha )\underset{s<t<0}{\inf }\left\{ \inf \left\{ \gamma >0;%
\text{ }ta>B_{\mu }((\gamma -t))\right\} \right\} .
\end{equation*}%
According to proposition \ref{prop4sym}, we verify that 
\begin{equation*}
\underset{s<t<0}{\inf }\left\{ \inf \left\{ \gamma >0;\text{ }ta>B_{\mu
}((\gamma -t))\right\} \right\} =\dfrac{1}{a}\underset{\theta \geq 1}{\inf }%
\left( \alpha \theta +B_{\mu }(\theta )\right) .
\end{equation*}%
It comes that%
\begin{equation*}
\underset{q\in E}{\inf }\left\{ \Phi _{q}T_{\chi }^{q}(\alpha )\right\} \leq 
\dfrac{T_{\chi }^{F}(\alpha )}{a}\underset{\theta \geq 1}{\inf }\left(
\alpha \theta +B_{\mu }(\theta )\right) .
\end{equation*}%
On the other hand, as $a<1$ and by proposition \ref{prop5symbis}, it follows
that $B_{\mu -}^{^{\prime }}(1)\leq -a$, therefore 
\begin{equation*}
\underset{\theta \geq 1}{\inf }\left( a\theta +B_{\mu }(\theta )\right) =%
\underset{\theta \geq 0}{\inf }\left( a\theta +B_{\mu }(\theta )\right) .
\end{equation*}%
Then we deduce according to proposition \ref{thv1sym}\ 
\begin{equation*}
\underset{q\in E}{\inf }\left\{ \Phi _{q}T_{\chi }^{q}(\alpha )\right\} \leq 
\dfrac{T_{\chi }^{F}(\alpha )}{a}\underset{q\in E}{\inf }(\left\langle
q,\alpha \right\rangle +B_{\chi }(q)).
\end{equation*}%
Remains to show that%
\begin{equation*}
\dfrac{T_{\chi }^{F}(\alpha )}{a}<1.
\end{equation*}%
Let $A\subset X_{\left\langle q,\alpha \right\rangle }(\eta ,p)$ and $%
(B(x_{i},r_{i}))$ a centered $\varepsilon -$packing of $A.$ It is clear that
for all $i\in I$, $x_{i}\in \mathcal{C}.$ Then we consider the selected
cylinder $\left[ j\right] _{i}=B(x_{i},r_{i})$ such that $j\in \mathcal{A}%
^{n+1}$ and $\dfrac{1}{2^{n+1}}\leq r_{i}<\dfrac{1}{2^{n}}.$ \newline
We consider the partition $I_{1}\cup I_{2}$ of $I$ such that\newline
\begin{equation*}
I_{1}=\left\{ i\in I:\left[ j\right] _{i}\text{ is of type }T_{1}\right\} 
\text{ and }I_{2}=I\backslash I_{1}.
\end{equation*}%
We recall that each cylinder $\left[ j\right] _{i}$, $i\in I_{2}$ is related
to a single cylinder type $T_{1}$ centered $x_{i}^{\prime }\in A$, denoted $%
\left[ j^{\prime }\right] _{i}$. With the condition of separation, $\left(
B\left( x_{i}^{\prime },\dfrac{1}{2^{n+1}}\right) \right) _{i\in I_{2}}$ is
a centered $\varepsilon -$packing of $A$. Then we consider the family $%
\left( B(y_{i},\delta _{i})\right) _{i\in I}$ defined by 
\begin{equation*}
B(y_{i},\delta _{i})=\left\{ 
\begin{array}{l}
B(x_{i},r_{i})\text{, }i\in I_{1} \\ 
B\left( x_{i}^{\prime },\dfrac{1}{2^{n+1}}\right) ,\text{ }i\in I_{2}.%
\end{array}%
\right.
\end{equation*}%
We verify that 
\begin{equation*}
\frac{\log \mu (B(y_{i},\delta _{i}))}{\log \delta _{i}}\leq \frac{\log \mu
\left( \left[ j\right] _{i}\right) }{\log \left( \dfrac{1}{2^{n+1}}\right) },%
\text{ }i\in I_{1}
\end{equation*}%
and 
\begin{equation*}
\frac{\log \mu (B(y_{i},\delta _{i}))}{\log \delta _{i}}\leq \frac{\log \mu
\left( \left[ j^{\prime }\right] _{i}\right) }{\log \left( \dfrac{1}{2^{n+1}}%
\right) },\text{ }i\in I_{2}.
\end{equation*}%
Given (\ref{eq19}) and the fact that the function $g$ is increasing, we
deduce that for all $i\in I,$ 
\begin{equation*}
\frac{\log \mu (B(y_{i},\delta _{i}))}{\log \delta _{i}}\leq g(\gamma _{1}).
\end{equation*}%
Yet for $q\in F$ and\ $\lambda >0,$ there exists $r_{0}>0$ such that $%
\varepsilon \leq r_{0}$ and $\delta _{i}<r_{0}$ 
\begin{equation*}
\frac{\left\langle q,\chi (y_{i},\delta _{i})\right\rangle }{\log \delta _{i}%
}\leq \lambda +\frac{\log \mu (B(y_{i},\delta _{i}))}{\log \delta _{i}},
\end{equation*}%
then 
\begin{equation*}
L_{\varepsilon ,(B(x_{i},r_{i}))_{i\in I}}^{q,2}(A)\leq \lambda +g(\gamma
_{1}).
\end{equation*}%
It follows that $L_{\varepsilon }^{q,2}(A)\leq \lambda +g(\gamma _{1}),$
letting $\varepsilon \rightarrow 0$ and $\lambda \rightarrow 0$ , we deduce
that 
\begin{equation*}
L^{q,2}(A)\leq g(\gamma _{1}).
\end{equation*}%
The sequence $\left( L^{q,k}(A)\right) _{k}$ is decreasing, it results that 
\begin{equation*}
L^{q}(A)\leq g(\gamma _{1}),
\end{equation*}%
then, 
\begin{equation*}
T_{\chi }^{q}(\alpha )\leq g(\gamma _{1}),
\end{equation*}%
as $g(\gamma _{1})<a$ and $T_{\chi }^{q}(\alpha )=T_{\chi }^{F}(\alpha )$ ,
it follows that 
\begin{equation*}
T_{\chi }^{F}(\alpha )<a.
\end{equation*}%
Finally we obtain 
\begin{equation*}
\underset{q\in E}{\inf }\left\{ \Phi _{q}T_{\chi }^{q}(\alpha )\right\} <%
\underset{q\in E}{\inf }(\left\langle q,\alpha \right\rangle +B_{\chi }(q)).
\end{equation*}
\end{proof}

\begin{corollary}
\ 
\begin{equation*}
\Dim(X_{\chi }\left( \alpha ,E\right) )\leq \underset{q\in E}{\inf }\left\{
\Phi _{q}T_{\chi }^{q}(\alpha )\right\} <\underset{q\in E}{\inf }%
(\left\langle q,\alpha \right\rangle +B_{\chi }(q)).
\end{equation*}
\end{corollary}

\begin{proof}
Follows from theorem \ref{th3v} and proposition \ref{prop7sym}.
\end{proof}

\bigskip

\end{document}